\documentclass[a4paper,10pt,reqno]{amsart}
\usepackage{amssymb}
\usepackage{xypic}
\usepackage[all,cmtip]{xy}
\usepackage{amsmath}

\allowdisplaybreaks

\DeclareMathOperator{\charac}{char}

\numberwithin{equation}{section}
\newtheorem{thm}{Theorem}[section]
\newtheorem{prop}[thm]{Proposition}
\newtheorem{lem}[thm]{Lemma}
\newtheorem{cor}[thm]{Corollary}
\theoremstyle{definition}
\newtheorem{defn}[thm]{Definition}
\newtheorem{rk}[thm]{Remark}

\newtheorem{ex}[thm]{Example}

\begin{document}
\title[Co-Poisson structures on polynomial Hopf algebras]{Co-Poisson structures on polynomial Hopf algebras}
\author{Qi Lou}
\author{Quanshui Wu}

\address{School of Mathematical Sciences, Fudan University, Shanghai 200433, China}
\email{qlou@fudan.edu.cn}

\address{School of Mathematical Sciences, Fudan University, Shanghai 200433, China}
\email{qswu@fudan.edu.cn}


\begin{abstract}
The Hopf dual $H^\circ$ of any Poisson Hopf algebra $H$
is proved to be a co-Poisson Hopf algebra provided $H$ is noetherian.
Without noetherian assumption, unlike it is claimed in literature, the statement does not hold.
It is proved that there is no nontrivial Poisson Hopf structure on the universal enveloping algebra of a non-abelian Lie algebra.
So the polynomial Hopf algebra, viewed as the universal enveloping algebra of a finite-dimensional abelian Lie algebra, is considered.
The Poisson Hopf structures on polynomial Hopf algebras are exactly linear Poisson structures.
The co-Poisson structures on polynomial Hopf algebras are characterized. Some correspondences between co-Poisson and Poisson structures are also established.
\end{abstract}

\maketitle

\section{Introduction}

Poisson  structure naturally appears in classical/quantum mechanics, in
mathematical physics, and
in deformation theory. It is an important algebra structure in Poisson
geometry, algebraic geometry and non-commutative geometry. There are lots of research in the related subjects.

Co-Poisson structure is a dual concept of Poisson structure in
categorial point of view. It arises also in mathematics and mathematical physics naturally as explained in the next two paragraphs.

Let $G$ be a Lie group and $\mathcal{O}(G)$ be its algebra of functions.
A  Lie group $G$ is said to be a Poisson Lie group if  $\mathcal{O}(G)$ is a Poisson Hopf algebra.
It is well-know that the category of connected and simply-connected Lie groups is equivalent to the category of finite-dimensional Lie algebras.
In this case, $\mathcal{O}(G)$ is identified with the Hopf dual $U(\mathfrak{g})^\circ$ of the universal enveloping algebra $U(\mathfrak{g})$, where $\mathfrak{g}$
is the corresponding Lie algebra of $G$.
The Poisson counterpart of this fact holds also, namely, the category of connected and simply-connected
Poisson Lie groups is equivalent to the category of finite-dimensional Lie bialgebras (\cite[Theorem 3.3.1]{KS} or \cite[Theorem 1]{Dr}).
However, the Lie bialgebra structures on any Lie algebra $\mathfrak{g}$  is in one-to-one correspondence with the co-Poisson Hopf structures on $U(\mathfrak{g})$ (\cite[Proposition 6.2.3]{CP}).

On the other hand, to quantize a Lie group or Lie algebra one should equip it with an extra structure, namely, a Poisson Lie group structure or Lie bialgebra structure, respectively.
Therefore co-Poisson structure naturally appears in the theory of quantum groups and in mathematical physics.

If $G$ is a connected and simply-connected Poisson Lie group and $\mathfrak{g}$ is the corresponding Lie bialgebra,
then the Poisson Hopf structure on $U(\mathfrak{g})^\circ \cong \mathcal{O}(G)$ is the dual of the  co-Poisson Hopf structures on $U(\mathfrak{g})$.
In \cite{OP}, the authors proved that the dual Hopf algebra $U(\mathfrak{g})^\circ$ of $U(\mathfrak{g})$ is a Poisson Hopf algebra for any finite-dimensional
Lie bialgebra $\mathfrak{g}$. In fact, in general, as stated in \cite[Proposition 3.1.5]{KS} earlier,  the dual Hopf algebra of any co-Poisson Hopf algebra is a Poisson Hopf algebra.
A complete proof  is given in a recent paper by Oh \cite[Theorem 2.2]{Oh}. The dual proposition that the dual Hopf algebra of any Poisson Hopf algebra is a co-Poisson Hopf algebra is also stated in \cite[Proposition 3.1.5]{KS}.
Unfortunately, this claim is not true in general as showed in our Example \ref{counter-ex}. Under an additional assumption that the algebra is noetherian,
we prove the statement is true in Proposition \ref{dual-of-coPoisson-Hopf}.

We prove that there is no nontrivial Poisson Hopf structure on the universal enveloping algebra of a non-abelian Lie algebra in Proposition \ref{U-of-non-abel}.
So, we turn to consider in latter sections the abelian case, i.e., the  polynomial Hopf algebra $A=k[x_1, x_2, \cdots, x_d]$, viewed as the universal enveloping algebra
of an abelian Lie algebra of dimension $d$.
The Poisson Hopf structures on $A=k[x_1, x_2, \cdots, x_d]$
are exactly linear Poisson structures on $A$ (see Proposition \ref{poisson-hopf-str}).
By establishing a reciprocity law between two linear maps of $A \to A \otimes A$ for $A=k[x_1, x_2, \cdots, x_d]$ (see Proposition \ref{main2}),
all co-Poisson coalgebra and co-Poisson Hopf algebra structures on  $A=k[x_1, x_2, \cdots, x_d]$
are described in Theorem \ref{main-result} and Proposition \ref{main3} respectively. In particular, the co-Poisson coalgebra structures on $A=k[x, y]$ are given by the linear maps $I: A \to k(x \otimes y - y \otimes x)$ (see Proposition \ref{2-variable-case}). By using the algebra of divided power series, the co-Poisson coalgebra structures on $A=k[x_1, x_2, \cdots, x_d]$ are showed to be in one-to-one correspondence with the Poisson algebra structures on $\tilde{A}=k[[x_1, x_2, \cdots, x_d]]$, the algebra of formal power series in Theorem \ref{main5}.


The paper is organized as follows. The definitions of (co-)Poisson (co)algebras are recalled in Section 2. Some preliminary results and examples are also given in Section 2.
In Section 3, we establish some dual properties between co-Poisson structures and Poisson structures.
In Section 4, we characterize co-Poisson coalgebra structures on polynomial Hopf algebra.
In Section 5, we characterize co-Poisson Hopf structures on polynomial Hopf algebras.


{\bf Convention.} Let $k$ be a base field. All vector spaces,
algebras, coalgebras, and Hopf algebras are over $k$. All linear
maps mean $k$-linear. Unadorned $\otimes$ means $\otimes_k$.

Let $V$ be a vector space. Let $t_n: V^{\otimes n} \to V^{\otimes
n}\, (n \in \mathbb{N}^+)$ be the linear map given by $v_1 \otimes
\cdots \otimes v_n \mapsto v_n \otimes v_1 \otimes \cdots \otimes
v_{n-1}.$ For convenience, let $\circlearrowleft=1+t_3+t_3^2$.

Suppose $(C, \Delta, \varepsilon)$ is a coalgebra where $\Delta$ is
the comultiplication and $\varepsilon$ is the counit. We frequently
use the sigma notation
$$\Delta(c)= \sum c_1 \otimes c_2 \, \, \textrm{and}\, \,
(\Delta \otimes 1) \Delta (c) =\sum c_1 \otimes c_2 \otimes c_3,
$$ where $\sum$ is often omitted in the computations.

Let $\Delta^{(2)}=(\Delta \otimes 1) \circ \Delta =(1 \otimes
\Delta) \circ \Delta : C \to C \otimes C \otimes C$, and $\Delta' =
\Delta - t_2 \circ \Delta$ be the cocommutator.

Suppose $(A, \mu, \eta)$ is an algebra where $\mu$ is the
multiplication and $\eta$ is the unit. For any $a,b \in A$,
$[a,b]=ab-ba$ is the commutator.

\section{Poisson structures and co-Poisson structures}

%
\subsection{Poisson algebras and Poisson Hopf algebras}
\begin{defn}\label{defn-poisson-alg} \cite{Li, Wei}
An algebra $A$ equipped with a linear map $\{-, -\} : A
\otimes A \to A$ is called a {\it Poisson algebra} if
\begin{enumerate}
\item $A$ with $\{-, -\} : A \otimes A \to A$  is a Lie algebra.
\item $\{-, c\} : A  \to A$  is a derivation with respect to the multiplication of $A$ for all $c \in A$, that is,
$\{ab, c\}=a\{b, c\}+\{a, c\}b$  for all $a, b \in A$.
\end{enumerate}
\end{defn}

It should be noted that we don't assume that $A$ is commutative in general.
As showed in \cite[Theorem 1.2]{FL}, if $A$ is prime and not commutative, then any nontrivial
Poisson structure $\{-, -\}$ on $A$ is the commutator bracket $[-,-]$ up to some scalar $\alpha$ in the center of the Martindale right quotient ring of $A$,
where $\alpha$ is essentially determined by a bimodule morphism from some nonzero ideal to $A$.

\begin{prop} \label{non-comm-poi}\cite[Theorem 1.2]{FL}
Let $A$ be a Poisson algebra with $\{A,A\} \neq 0$ and $[A,A] \neq 0$.
If $A$ is prime, then
\begin{enumerate}
\item for any $a,b \in A$, there is an isomorphism
of $A$-bimodules
$$f_{(a,b)}: A[a,b]A \to A \{a,b\}A, [a,b] \mapsto \{a,b\};$$
\item in the Martindale right quotient ring of $A$, $\{-, -\}=\alpha [-, -]$,  where $\alpha$ is represented by
 $f_{(a,b)}$ for any $a, b \in A$ such that $[a, b] \neq 0.$
\end{enumerate}
\end{prop}

\begin{proof} (1)
By using Leibniz rule and calculating $\{ac,bd\}$ in both variables (see \cite[Theorem 1]{Vo} or  \cite[Lemma 1.1]{FL}), it follows that,
for all $a, b, c, d \in A$,
\begin{equation} \label{abcd}
[a,b]\{c,d\}=\{a,b\}[c,d].
\end{equation}
Then, for all $a, b, c, d, e \in A$,
$[a,b]\{ec,d\}=\{a,b\}[ec,d]$, which implies that
\begin{equation} \label{abcde}
[a,b]e\{c,d\}=\{a,b\} e [c,d].
\end{equation}

Since $[A,A] \neq 0$, $\{A,A\} \neq 0$ and $A$ is
prime, then $[a,b]=0$ if and only if $\{a,b\}=0$ for any $a,b \in
A$.

Suppose $[a,b] \neq 0$. If $\sum_{i=1}^n x_i [a,b] y_i=0$, then
for any $c \in A$,
$$0=\sum_{i=1}^n (\{a,b\}c)x_i [a,b] y_i =\sum_{i=1}^n[a,b]c(x_i\{a,b\}y_i).$$
Then $\sum_{i=1}^n x_i \{a,b\}
y_i=0$ as $A$ is prime and $[a,b] \neq 0$. This shows that the map
$$f_{(a,b)}: A [a,b] A \to A\{a,b\} A,~~ \sum_{i=1}^n x_i [a,b] y_i
\mapsto \sum_{i=1}^n x_i \{a,b\} y_i$$ is well-defined.
So, $f_{(a,b)}$ is an
$A$-bimodule morphism, and is in fact an isomorphism.

(2) If $[a, b] \neq 0$, then the bimodule morphism $f_{(a,b)}: A [a,b] A \to A, [a,b] \mapsto \{a,b\}$
represents an element in the center of the Martindale right quotient ring of $A$ (\cite[10.3.5]{MR} or \cite[Chapter 1, \S 3]{He}).
It follows from \eqref{abcde} that  $f_{(a,b)}$ and $f_{(c,d)}$ represent the same element, say $\alpha$, in the Martindale right quotient ring of $A$ if both  $[a, b]$ and  $[c, d]$
are nonzero. Then $\{-, -\}=\alpha [-, -]$ in the Matindale right quotient ring of $A$.
\end{proof}

By Proposition \ref{non-comm-poi}, any
Poisson algebra structure $\{-, -\}$ on the Weyl algebra $A_n(k)$ or the matrix algebra $M_n(k)$ is $\alpha[-,-]$ for some $\alpha\in k$.

Commutative Poisson algebras appear naturally in geometry and algebra.
For more examples of commutative Poisson algebras, see \cite{LPV}, \cite{KS} and \cite{LWW}.

\begin{defn}(\cite[Definition 6.2.1]{CP})
Let $(H, \mu, \eta,\Delta,\varepsilon,S)$ be a Hopf algebra with a
linear map $\{-,-\}: H \otimes H \to H$. Then $H$ is called a {\it
Poisson Hopf algebra} if
\begin{enumerate}
\item $(H,\{-,-\})$ is a Poisson algebra.
\item The structures are compatible in the sense that, for all $a,b \in H$,
\begin{align}\label{poisson-hopf}\Delta \left( \{a,b\} \right)
=\sum \{a_1,b_1 \} \otimes a_2b_2 + \sum a_1b_1 \otimes
\{a_2,b_2\}.\end{align}
\end{enumerate}
\end{defn}

Let $A$ and $B$ be two Poisson algebras. An algebra morphism $f: A
\to B$ is called a {\it Poisson algebra morphism} if $f \left(
\{a,b\}_A\right) = \{ f(a), f(b)\}_{B}$ for all $a,b \in A$.

Equation \eqref{poisson-hopf} means that $\Delta : H \to H \otimes H$ is a Poisson
algebra morphism when $H$ is commutative by the following lemma.

\begin{lem}
Let $A$ and $B$ be two commutative Poisson algebras. There is
naturally a Poisson structure $\{-, -\}_{A \otimes B}$ on the tensor
product algebra $A \otimes B$ given by
$$\{a \otimes b , a' \otimes b'\}_{A \otimes B}=\{a,a'\}_A \otimes bb' + aa' \otimes \{b,b'\}_B.$$
\end{lem}

It is easy to see that there is no nontrivial Poisson Hopf algebra structure on any group algebra $k[G]$.
There is also no nontrivial Poisson Hopf structure on the universal enveloping algebra $U(\mathfrak{g})$
if $\mathfrak{g}$ is a non-abelian Lie algebra.

\begin{prop} \label{U-of-non-abel}
Let $\mathfrak{g}$ be a non-abelian Lie algebra over a field of characteristic $\neq 2$. Then there is no nontrivial Poisson Hopf structure on $U(\mathfrak{g})$.
\end{prop}

\begin{proof}
Suppose $\{-,-\}: U(\mathfrak{g}) \otimes U(\mathfrak{g}) \to U(\mathfrak{g})$ is a non-trivial Poisson Hopf structure on $U(\mathfrak{g})$.
Since $\mathfrak{g}$ is non-abelian and $U(\mathfrak{g})$ is prime, then by Proposition \ref{non-comm-poi}, for any $a, b \in  \mathfrak{g}$,  $[a, b]=0$ if and only if $\{a, b\}=0$.
Recall the compatible condition \eqref{poisson-hopf}, i.e.,
 for all $x,y \in U(\mathfrak{g})$,
\[\Delta \{x,y\} = x_1 y_1 \otimes \{x_2,y_2\} + \{x_1,y_1\} \otimes x_2 y_2.\]
This implies that if both $x$ and $y$ are primitive then so is $\{x,y\}$.
If we take $x=ab$ and $y =cd$ for  $a,b,c,d \in   \mathfrak{g}$, the  compatible condition \eqref{poisson-hopf} will imply that
\[ [b,c] \otimes \{a,d\} + [b,d] \otimes \{a,c\}  \]
is skew-symmetric for all $a,b,c,d \in \mathfrak{g}$.
So, $[b,c] \otimes \{a,b\}$ is skew-symmetric by taking $d =b$.
Suppose $[b,c ] \neq 0$ as $[ \mathfrak{g}, \mathfrak{g}] \neq 0$. Since both $[b,c]$ and $\{a,b\}$ are primitive,
then  $\{ a, b\}=0$ for all $a \in \mathfrak{g}$ as $\charac k \neq 2$. In particular, $\{c, b\}=0$, and so, $[c, b]=0$, which is a contradiction.
\end{proof}

As pointed out in \cite[Remark 3.1.4]{KS}, if $H$ is a commutative
Poisson Hopf algebra, then the counit $\varepsilon: H \to k$ is a
Poisson algebra morphism, and the antipode $S: H \to H$ is a Poisson
algebra anti-morphism. Here is a proof of the facts (see also \cite[Lemma 4.2]{Oh}).

\begin{lem} \label{anti-poisson-morphism} Let $H$ be a Poisson Hopf algebra.
Then the counit $\varepsilon: H \to k$ is a Poisson algebra morphism.
The antipode $S: H \to H$ is a Poisson algebra anti-morphism provided $H$ is commutative.
\end{lem}

\begin{proof} We need to show $\varepsilon \left( \{a,b\}\right)=0$ and $S \left( \{a,b\}\right)
= \{ S(b),S(a)\}$ for all $a,b \in H$.
Since $\varepsilon(h)=\varepsilon(h_1) \varepsilon(h_2)$ for all $h \in H$, by \eqref{poisson-hopf},
\[\varepsilon \left( \{a,b\}\right)= \varepsilon \left( \{a_1,b_1\}\right) \varepsilon (a_2b_2) +
\varepsilon (a_1b_1)\varepsilon \left( \{a_2,b_2\}\right)=2
\varepsilon \left( \{a,b\}\right).\] Thus $\varepsilon \left(
\{a,b\}\right)=0$. For the second assertion,
\begin{align*}
 \{ S(b),S(a)\} =& \{S(b_1), S(a)\} b_2 S(b_3)\\
 =&-S(b_1) \{b_2, S(a)\} S(b_3)\\
 =&-S(b_1) \{b_2, S(a_1)\}a_2 S(a_3) S(b_3)\\
 =&S(b_1) S(a_1) \{ b_2,a_2\} S(a_3) S(b_3)\\
 =&-S(b_1) S(a_1) b_2a_2 S \left(  \{b_3, a_3\}\right)\\
 =&S\left(\{a, b\}\right),
\end{align*}
where the second last $``="$ follows from the fact $ \{a_1,b_1\}
S(a_2b_2) + a_1 b_1S \left( \{a_2,b_2\}\right) = 0 $ for all $a,b
\in H$, and the last $``="$ follows from the commutativity of $H$.
\end{proof}

\subsection{Co-Poisson coalgebras and co-Poisson Hopf algebras}
The concepts of co-Poisson coalgebra and co-Poisson Hopf algebra are
introduced in \cite[Chapter 6]{CP} to study quantizations of Lie
bialgebras. They are dual to Poisson algebra and Poisson Hopf algebra.

\begin{defn}\label{defn-copoisson-coalg} (\cite[Definition 6.2.2]{CP})
A coalgebra $(C, \Delta, \varepsilon)$ equipped with a linear map $q: C \to C \otimes C$ is
called a {\it co-Poisson coalgebra} if
\begin{enumerate}
\item $C$ with $q: C \to C \otimes C$ is a Lie coalgebra, that is,
\begin{align*}
(1+t_2) \circ q & =0, & \textrm{ (skew-symmetric)}\\
(1+t_3+t_3^2) \circ (q \otimes 1) \circ q & =0. & \textrm{ (co-Jacobi
identity)}
\end{align*}
\item$(\Delta \otimes 1) q = (1 \otimes q) \Delta -t_3^2 (q \otimes 1) \Delta$. \qquad \qquad \quad(co-Leibniz rule)
\end{enumerate}
\end{defn}

In co-Poisson coalgebra $C$, we use the sigma notations
$$\Delta(c)=\sum c_1 \otimes c_2, \,  q(c) =\sum c_{(1)} \otimes c_{(2)} \, \, \textrm{and} \, \, (q \otimes 1) q(c)
=\sum c_{(1)} \otimes c_{(2)} \otimes c_{(3)} $$ where $\sum$ is often omitted in the computations. Then,
$$(1 \otimes q) q(c) = (1 \otimes q) (   -c_{(2)} \otimes c_{(1)})= - c_{(3)} \otimes c_{(1)}
\otimes c_{(2)}= c_{(3)} \otimes c_{(2)} \otimes c_{(1)}.$$

\begin{rk}
By using the sigma notation, the co-Leibniz rule reads as
$$ \sum c_{(1)1} \otimes c_{(1)2} \otimes c_{(2)}
= \sum c_1 \otimes c_{2(1)} \otimes c_{2(2)}-\sum c_{1(2)} \otimes
c_2 \otimes c_{1(1)}$$ for all $c \in C$. It is equivalent to
\begin{equation}\label{co-Leibnitz-1}
(1 \otimes \Delta)q=(q \otimes 1) \Delta -t_3 (1 \otimes q) \Delta.
\end{equation}
If the coalgebra $C$ is cocommutative, then the co-Leibniz rule is
also equivalent to
\begin{equation}\label{co-Leibnitz-2}
(\Delta \otimes 1) q =(1-t_3) (1 \otimes q) \Delta.
\end{equation}
\end{rk}

The cocommutator $\Delta'$ gives a co-Poisson structure on any coalgebra $(C, \Delta, \varepsilon)$.

It follows from
the co-Leibniz rule  that there is no nontrivial co-Poisson coalgebra structure on any group coalgebra $k[G]$ .

\begin{defn} Let $(C, \Delta, \varepsilon)$ be a coalgebra. A linear map $d: C \to C$ is called a {\it coderivation} if, for all $c \in C$,
\[\Delta (d(c)) = d(c_1) \otimes c_2 + c_1 \otimes d(c_2).\]
\end{defn}

\begin{ex}
Let $C$ be a cocommutative coalgebra, $d_1$ and $d_2$ be two coderivations of $C$ such that $d_1d_2=d_2d_1$. Let
$q: C \to C \otimes C$ be the map $c \mapsto q(c)$, where
\[q(c)= d_1(c_1) \otimes d_2 (c_2) -d_2(c_1) \otimes d_1(c_2).\]
Then $(C,q)$ is a co-Poisson coalgebra.
\end{ex}

Dual to the fact $\{a, -\}: A \to A$ is a derivation in any Poisson algebra
$A$ is the following.

\begin{lem}
Let $(C, \Delta, \varepsilon,q)$ be a co-Poisson coalgebra. Then for any $f \in C^*$, $ (f \otimes 1) q : C \to C$ and $ (1 \otimes f) q : C \to C$ are coderivations of $C$.
\end{lem}

Dual to the fact $\{1_A, a\}=\{a, 1_A\}=0$ in any Poisson algebra
$A$ is the following.
\begin{lem} \label{co-poisson-fact}
Let $(C, \Delta, \varepsilon,q)$ be a co-Poisson coalgebra. Then $$
(\varepsilon \otimes 1)  q=   (1 \otimes \varepsilon )
q=0, \, \textrm{that is},\, \varepsilon (c_{(1)}) c_{(2)} =c_{(1)} \varepsilon
(c_{(2)})=0 \, \, \textrm{for all}\, \,c \in C.$$
\end{lem}


\begin{defn}\label{defn-copoisson-Hopf} \cite[Definition 6.2.2]{CP}
A Hopf algebra $(H, \mu,\eta,\Delta,\varepsilon,S)$ equipped with a
linear map $q: H \to H \otimes H$ is called a {\it co-Poisson Hopf
algebra} if
\begin{enumerate}
\item $H$ with $q: H \to H \otimes H$ is a co-Poisson coalgebra.
\item $q$ is a $\Delta$-derivation, that is, for all $a, b \in H$,
\begin{equation}\label{q(ab)} q(ab) = q(a)\Delta(b)+ \Delta(a)q(b).\end{equation}
\end{enumerate}
\end{defn}

%

Co-Poisson Hopf structures on $U(\mathfrak{g})$ are known.

\begin{prop} \cite[Proposition 6.2.3]{CP} Let $\mathfrak{g}$ be a Lie algebra over a field $k$ of characteristic zero. Then the co-Poisson Hopf structures on the
universal enveloping algebra $U(\mathfrak{g})$ are determined uniquely by the Lie bialgebra structures on $\mathfrak{g}$.
\end{prop}

Let $C$ and $D$ be two co-Poisson coalgebras. A coalgebra morphism
$g: C \to D$ is called  a {\it co-Poisson coalgebra morphism} if $(g
\otimes g) (q_C(c)) =q_D \left(g(c)\right)$ for all $c \in C$.

Equation \eqref{q(ab)} means that $\mu : H \otimes H \to H$ is a co-Poisson
coalgebra morphism provided $H$ is cocommutative by Lemma
\ref{tensor-product-of-coalg}.

If $C$ and $D$ are two coalgebras, then $C \otimes D$ is a coalgebra
with $$\Delta_{C \otimes D}=(1 \otimes \tau \otimes 1) (\Delta_C
\otimes \Delta_D),\, \textrm{ i.e., }\,
\Delta_{C \otimes D}( a
\otimes b)=a_1 \otimes b_1 \otimes a_2 \otimes b_2.$$

\begin{lem} \label{tensor-product-of-coalg} Let $(C, q_C)$ and $(D, q_D)$ be two cocommutative co-Poisson
coalgebras. Then $(C \otimes D, q_{C \otimes D})$ is a co-Poisson
coalgebra, with $q_{C \otimes D}$ being the composition
$$C \otimes D \xrightarrow{q_C \otimes \Delta_D + \Delta_C \otimes q_D}
C \otimes C \otimes D \otimes D \xrightarrow{1 \otimes t_2 \otimes
1} C \otimes D \otimes C \otimes D,$$ that is,
$$q_{C \otimes D} (a \otimes b)=a_{(1)} \otimes b_1 \otimes a_{(2)} \otimes b_2 +
a_1 \otimes b_{(1)} \otimes a_2 \otimes b_{(2)}.$$
\end{lem}

The dual form of Lemma \ref{anti-poisson-morphism} is the following.
\begin{lem}
Let $(H, \mu , \eta, \Delta ,\varepsilon , S, q)$ be a co-Poisson
Hopf algebra.
\begin{enumerate}
\item $\eta$ is a co-Poisson coalgebra morphism.
\item If $H$ is cocommutative, then $S$ is a co-Poisson coalgebra anti-morphism.
\end{enumerate}
\end{lem}

\begin{proof}
(1) This is trivial as $q(1_H)=0$ by \eqref{q(ab)}.

(2) We need to show $q\left(S(h)\right)=S(h_{(2)}) \otimes S(h_{(1)})$ for
all $h \in H$.
It follows from the
co-Leibniz rule that for any $h \in H$,
\[h_{1(1)1} \otimes h_{1(1)2} \otimes h_2 \otimes h_{1(2)}=h_1 \otimes h_{2(1)} \otimes h_3 \otimes h_{2(2)}
-h_{1(2)} \otimes h_2 \otimes h_3 \otimes h_{1(1)}.\]
Since
$S(h_{1(1)1}) h_{1(1)2} S(h_2) \otimes S(h_{1(2)})=0$ by Lemma \ref{co-poisson-fact}, then
$$S(h_1)h_{2(1)} S(h_3) \otimes S(h_{2(2)}) = S(h_{1(2)}) \varepsilon
(h_2) \otimes S(h_{1(1)}) = S(h_{(2)}) \otimes S(h_{(1)}),$$ where the
last $``="$ is derived from the co-Leibniz rule \eqref{co-Leibnitz-1} and Lemma
\ref{co-poisson-fact}.

On the other hand, by the co-Leibniz rule,
\begin{align*}
&h_1 \otimes h_2 \otimes h_{3(1)1} \otimes h_{3(1)2} \otimes h_{3(2)} \otimes h_4 \\
=&h_1 \otimes h_2 \otimes h_3 \otimes h_{4(1)} \otimes h_{4(2)} \otimes h_5 -h_1 \otimes h_2 \otimes h_{3(2)} \otimes h_4 \otimes h_{3(1)} \otimes h_5.
\end{align*}
By Lemma \ref{co-poisson-fact},
$S(h_1) h_{3(2)} S(h_4) \otimes S(h_2) h_{3(1)1} S(h_{3(1)2})=0.$
Thus
\begin{align*}&S(h_1)h_{3(1)} S(h_5) \otimes S(h_2) h_{3(2)} S(h_4)\\
=& S(h_1) h_{4(2)} S(h_5) \otimes S(h_2) h_3 S(h_{4(1)})\\
=&S(h_1) h_{2(2)} S(h_3) \otimes S(h_{2(1)})\\
=&S(h_{(1)}) \otimes S(h_{(2)}).
\end{align*}

It remains to show that $S(h_1)h_{3(2)} S(h_5) \otimes S(h_2) h_{3(1)} S(h_4) = q \left(S(h) \right)$.
Since $0=q(h_1 S(h_2))=q(h_1) \Delta \left( S(h_2)\right) + \Delta (h_1) q \left(  S(h_2)\right)$ for any $h \in H$, then
\[h_{1(1)} S(h_3) \otimes h_{1(2)} S(h_2) + h_1 S(h_3)_{(1)} \otimes h_2 S(h_3)_{(2)}=0.\]
Hence
\[h_1 \otimes  h_3 S(h_5)_{(1)} \otimes h_2 \otimes h_4 S(h_5)_{(2)} + h_1 \otimes h_{3(1)} S(h_5) \otimes h_2 \otimes h_{3(2)}S(h_4)=0.\]
Thus
\begin{align*} &S(h_1)h_{3(2)} S(h_5) \otimes S(h_2) h_{3(1)} S(h_4)\\
=&S(h_1) h_3 S(h_5)_{(1)} \otimes S(h_2) h_4 S(h_5)_{(2)}\\
=&S(h_1) h_2 S(h_5)_{(1)} \otimes S(h_3) h_4 S(h_5)_{(2)}\\
=&S(h)_{(1)} \otimes S(h)_{(2)}\\
=&q \left(  S(h)\right),
\end{align*}
where the cocommutativity is used for the second $``="$.
\end{proof}

The dual form of \eqref{abcd} for co-Poisson coalgebras is the following
 proposition.

\begin{prop} \label{dual-of-abcd} Let $(C, \Delta, \varepsilon)$ be a co-Poisson coalgebra with co-Poisson structure $q$. Then
$$(q \otimes \Delta')\circ \Delta =(\Delta' \otimes q)\circ\Delta.$$
\end{prop}

\begin{proof} It follows by using the co-Leibniz rule to calculate
$$(\Delta \otimes \Delta)q(a)=(\Delta \otimes 1 \otimes 1)(1 \otimes \Delta) q(a)=(1 \otimes 1 \otimes \Delta)(\Delta \otimes 1) q(a).$$
\end{proof}

The dual form of \eqref{abcde} for co-Poisson coalgebras is the following.
\begin{cor}
Let $(C, \Delta, \varepsilon)$ be a co-Poisson coalgebra with
co-Poisson structure $q$. Then
$$(\Delta' \otimes (1 \otimes q) \Delta) \circ \Delta=(q \otimes (1 \otimes \Delta') \Delta) \circ \Delta,$$
i.e, for any $a \in C$,
\begin{align*}&\,\,\,\,\,\,\,\,a_1 \otimes a_2 \otimes a_3 \otimes q(a_4)-a_2 \otimes a_1 \otimes a_3 \otimes q(a_4) \\
&=q(a_1) \otimes a_2 \otimes a_3 \otimes a_4 -q(a_1) \otimes a_2
\otimes a_4 \otimes a_3.\end{align*}
\end{cor}

Poisson and co-Poisson structures on 4-dimensional Sweedler Hopf algebra are clear.

\begin{ex}
Let $H_4= k \langle x, g \mid x^2=0, g^2=1, xg=-gx \rangle $ be the
4-dimensional Sweedler Hopf algebra, where the coalgebra structure
is given by $\Delta(g)=g \otimes g$, $\Delta(x)=x \otimes 1 + g
\otimes x$, $\varepsilon(g)=1$, $\varepsilon(x)=0$; and the antipode
is given by  $S(g)=g=g^{-1}$, $S(x)=-gx$. Assume $\charac k \neq
2$.
\begin{enumerate}
\item Any Poisson structure on $H_4$ is given by $\{g,x\}=\lambda x + \mu gx$ for some $\lambda, \mu \in k$.
\item There is no nontrivial Poisson Hopf structure on $H_4$. In
fact, if the Poisson structure given by $\{g,x\}=\lambda x + \mu gx$
is a Poisson Hopf structure, then by applying \eqref{poisson-hopf}
with $a=g, b=x$ and $a=x ,b=gx$ we get $\lambda=\mu=0$.

\item Any co-Poisson structure $q$ on $H_4$ is given by
\begin{align*}
q(1)&=q(g)=0,\\
q(x)&= \alpha ( 1 \otimes x - x \otimes 1 + x \otimes g - g \otimes x),\\
q(gx)&=\beta(1 \otimes gx -gx \otimes 1 +  gx \otimes g - g \otimes
gx)
\end{align*}
 for some $\alpha, \beta \in k$. In fact, for any $h \in H_4$, we may assume
\begin{align*}
q(h)&= \alpha_1 (1 \otimes x - x \otimes 1) + \alpha_2 (1 \otimes g -g \otimes 1)
+ \alpha_3 (1 \otimes gx -gx \otimes 1) \\
  & + \alpha_4 (x \otimes g - g \otimes x) +\alpha_5 (x \otimes gx
-gx \otimes x) + \alpha_6 (g \otimes gx -gx \otimes g)
\end{align*}
as $q(h)$ is skew-symmetric. It follows from $(\varepsilon \otimes
1) q=0$ (Lemma \ref{co-poisson-fact}) that
\[ \alpha_1=\alpha_4, \alpha_2=0, \alpha_3=-\alpha_6.\]
Then, the co-Leibniz rule implies that $q(1)=q(g)=0$,
\[q(x)= \alpha ( 1 \otimes x - x \otimes 1 + x \otimes g - g \otimes
x),\] \[q(gx)=\beta(1 \otimes gx -gx \otimes 1 +  gx \otimes g - g
\otimes gx)\] for some $\alpha, \beta \in k$.
\item There is no nontrivial co-Poisson Hopf structure on
$H_4$. In fact, suppose a co-Poisson structure $q$ as given in (3)
is a co-Poisson Hopf structure on $H_4$, then $\alpha=\beta=0$ by
the equations
$$q(gx)=q(g) \Delta (x)+ \Delta(g) q(x)\, \textrm{ and }\, 0=q(x^2)= q(x)
\Delta (x) + \Delta(x) q(x).$$
\end{enumerate}
\end{ex}


\section{Dual properties between Poisson and co-Poisson Hopf algebras}

As the vector space dual of any coalgebra is an algebra, the dual of
any co-Poisson coalgebra is a Poisson algebra.
\begin{prop}\label{dual-of-co-Poisson}
\noindent
\begin{enumerate}
\item Suppose $(C, \Delta, \varepsilon)$ is a coalgebra, $q: C \to C \otimes C$
is a linear map. Then $(C, q)$ is a co-Poisson coalgebra if and only
if $(C^*, q^*)$ is a Poisson algebra.
\item Suppose $\sigma: C \to D$ is a linear map between co-Poisson coalgebras. Then $\sigma: C \to D$ is a
co-Poisson coalgebra morphism  if and only if $\sigma^*: D^* \to
C^*$ is a Poisson algebra morphism.
\end{enumerate}
\end{prop}

\begin{proof} (1) Note that $q^*$ is the map
$q^*: C^* \otimes C^* \subseteq (C \otimes C)^* \to C^*,  f \otimes
g \mapsto \{f, g\}$, where $\{f, g\} \in C^*$ is the map $C \to k,
\, c \mapsto (f \otimes g)q(c)=f(c_{(1)})g(c_{(2)})$.

Note that $(\{f, g\} + \{g, f\})(c)= (f \otimes g)(1 + t_2)q(c)$. It is easy to see that $\{-,
-\}_{C^*}: C^* \times C^* \to C^*$ is skew-symmetric if and only if
$q$ is skew-symmetric.

Since $\{\{f,g\},h\}(c)=(f \otimes g \otimes h)(q \otimes
1)q(c)$, 
then
$$\circlearrowleft \{\{f,g\},h\}(c)= (f \otimes g \otimes
h)(1+t_3+t_3^2)(q \otimes 1)q(c).$$
It follows that $\{-, -\}_{C^*}$ satisfies the Jacobi identity if and
only if $q$ satisfies the co-Jacobi identity.

Similarly, since $\{f*g,h\}(c) =(f \otimes g \otimes h)\left((\Delta
\otimes 1) q(c) \right)$ and
$$((f* \{g,h\}) - (\{h,f\}*g) )(c) = (f \otimes g \otimes h) ((1 \otimes q) \Delta - t_3^2 (q \otimes 1)
\Delta)(c),$$ then $\{f*g,h\}=f* \{g,h\} + \{ f,h\} *g $ if and only
if $(\Delta \otimes 1) q = (1 \otimes q) \Delta-t_3^2 (q \otimes 1)
\Delta$.

(2) It is well known that $\sigma: C \to D$ is a coalgebra morphism
if and only if $\sigma^*: D^* \to C^*$ is an algebra morphism. Since
$$(\sigma^*  \{f,g\}_{D^*}  - \{\sigma^*f, \sigma^*g\}_{C^*})(c)=
(f \otimes g)(\Delta_D \sigma - (\sigma \otimes
\sigma)\Delta_C)(c),$$ then $\sigma^* \{-,-\}_{D^*}= \{-,-\}_{C^*}
(\sigma^* \otimes \sigma^*)$ if and only if $\Delta_D \sigma =
(\sigma \otimes \sigma)\Delta_C$. The proof is finished.
\end{proof}

If $A$ is an algebra, then $A^\circ =\{f \in A^* \mid \ker f \,
\textrm{contains a cofinite ideal of}\, A \}$ is a coalgebra, which
is called the finite dual of $A$. Suppose $(A, p=\{-,-\})$ is a
Poisson algebra. Example \ref{non-noetherian-ex} shows that
$(A^\circ, q=p^*)$ may not be a co-Poisson coalgebra because $p^*:
A^* \to (A \otimes A)^*$ may not be restricted to a map $p^*_{A^\circ}:
A^\circ \to A^\circ \otimes A^\circ$.

\begin{ex} \label{non-noetherian-ex}
Let $A=k[x_1, x_2, \cdots, x_n, \cdots]$ be a polynomial algebra
with infinitely many variables $\{x_i \mid i \geq 1\}$. Let $p(x_i \otimes x_j)=\{x_i,
x_j\}=1$ for all $i < j$. Then $p$ gives a Poisson algebra structure
on $A$. Suppose $\varepsilon: A \to k, x_i \mapsto 0$, is the
augmentation map. Then $\varepsilon \in A^{\circ}$, but
$p^*(\varepsilon) = \varepsilon p\notin (A \otimes A)^\circ$.

In fact, if $\varepsilon p\in (A \otimes A)^\circ$, then there is a
cofinite ideal $I$ of $A$ such that $I \otimes A + A \otimes I \subseteq
\ker (\varepsilon p)$. Then, $\{I, A\}\subseteq \ker \varepsilon$.
Since $A/I$ is finite-dimensional, there exists a nonzero linear
polynomial in $I$, which will imply that $1 \in \{I, A\}$. It
contradicts to $\{I, A\}\subseteq \ker \varepsilon$.
\end{ex}

%
%

But if $A$ is left or right noetherian, then we
have the following positive conclusion.

\begin{prop}  \label{A-o-copoisson} Let $(A, p=\{-,-\})$ be a Poisson algebra, and $f: A \to B$ be a Poisson algebra morphism.
\begin{enumerate}
\item If $A$ is a (left or right) noetherian algebra, then $(A^\circ, q=p^*)$ is a co-Poisson
coalgebra.
\item If both $A$ and $B$ are (left or right) noetherian, then $f^*:
B^\circ \to A^\circ$ is a co-Poisson coalgebra morphism.
\end{enumerate}
\end{prop}

\begin{proof} 
It suffices to show that $p^* : A^* \to (A \otimes A)^*$ restricts
to a map $$p^*|_{A^\circ}: A^\circ \to A^\circ \otimes A^\circ \cong
(A \otimes A)^\circ.$$

Suppose $f \in A^\circ$ and $I \subseteq \ker f$ is a cofinite ideal of
$A$. Since $A$ is left or right noetherian, then $I/ I^2$ is a finitely
generated left or right $A/I$-module.
Hence $\dim_k I/ I^2 < \infty$ as $\dim_k A/I < \infty$. So, $I^2$ is also a cofinite ideal of
$A$. Let $J=I^2 \otimes A + A \otimes I^2$. Then $J$ is a cofinite
ideal of $A \otimes A$, and
\[p^* (f) (J)=f  p (J) \subseteq f(I)=\{0\}.\]
It follows that $p^*(f) \in (A \otimes A)^\circ$.
\end{proof}

Note that $A^\circ = 0$ for the Weyl algebra $A=A_n(k)$. So, even in noetherian case $(A, p)$ may not be a Poisson algebra
when $(A^\circ, q=p^*)$ is a co-Poisson
coalgebra.

In \cite{OP}, the authors prove that the Hopf dual $H^\circ$ of a
co-Poisson Hopf algebra $H$ is a Poisson Hopf algebra when $H$ is an
almost normalizing extension over $k$. In fact, this is true in general
as stated in \cite[Proposition 3.1.5]{KS}. We give a proof here. Oh
gives a proof also in a recent paper (\cite[Theorem 2.2]{Oh}).

\begin{prop}
Let  $(H, \mu , \eta, \Delta ,\varepsilon , S, q)$ be a co-Poisson Hopf algebra. Then the Hopf dual $H^\circ$
is a Poisson Hopf algebra.
\end{prop}

\begin{proof} It is well-know that $H^\circ$
is a Hopf algebra (\cite[Section 6.2]{Sw}, \cite[Theorem
9.1.3]{Mo}). By Proposition \ref{dual-of-co-Poisson}, to show $H^\circ$ is
a Poisson algebra it suffices to show that $q^* : (H \otimes H)^* \to H^*$ restricts to a map $q^\circ : H^\circ \otimes H^\circ \cong
(H \otimes H)^\circ \to H^\circ,$ that is, $q^*\left( (H \otimes
H)^\circ \right) \subseteq H^\circ$.

Suppose $F \in (H \otimes H)^\circ$ and $J \subseteq \ker F$ is a
cofinite ideal in $H \otimes H$. Since $\Delta$ is an
algebra morphism, $\Delta^{-1}(J)$ is an ideal of $H$. It follows
from $q(ab)=q(a) \Delta(b)+\Delta(a) q(b)$ that $I= q^{-1}(J) \cap
\Delta^{-1} (J)$ is an ideal of $H$. Since the linear map
$$H/I \to \frac{H \otimes H} {J} \oplus \frac{H \otimes H} {J}, \quad
\overline{h} \mapsto \left( \overline{q(h)},
\overline{\Delta(h)}\right)$$ is injective, $I$ is a cofinite ideal
of $H$. Note that $q^*(F) (I) = F(q(I)) \subseteq F(J)=0$, i.e., $I
\subseteq \ker q^*(F)$. It follows that $q^*(F) \in H^\circ$.

To finish the proof, we need to show that $\mu^*: H^\circ \to
H^\circ \otimes H^\circ$ satisfies the compatible condition \eqref{poisson-hopf}, that is, for all $f, g \in H^\circ$,
$$\mu^* (\{f,g\}_{H^\circ}) =\{f_1,g_1\}_{H^\circ} \otimes (f_2 * g_2) + (f_1 * g_1) \otimes \{f_2, g_2\}_{H^\circ}.$$
This is true, because, for any
$x,y \in H$,  on one hand, by \eqref{q(ab)},
\begin{align*}
&\mu^* (\{f,g\}_{H^\circ})(x \otimes y)\\
=&\{f,g\}_{H^\circ}(xy)=(f \otimes g)q(xy)=f((xy)_{(1)}) g((xy)_{(2)})\\
=&f((x_{(1)}y_1) g((x_{(2)}y_2) + f((x_1y_{(1)}) g((x_2y_{(2)}),
\end{align*}
on the other hand,
\begin{align*}
&(\{f_1,g_1\}_{H^\circ} \otimes (f_2 * g_2) + (f_1 * g_1) \otimes \{f_2, g_2\}_{H^\circ})(x \otimes y)\\
=&(\{f_1,g_1\}_{H^\circ} (x) \, (f_2 * g_2)(y) + (f_1 * g_1) (x)  \, \{f_2, g_2\}_{H^\circ}(y)\\
=&f_1(x_{(1)}) g_1(x_{(2)}) f_2(y_1) g_2(y_2) + f_1(x_1) g_1(x_2) f_2(y_{(1)}) g_2(y_{(2)})\\
=&f((x_{(1)}y_1) g((x_{(2)}y_2) + f((x_1y_{(1)}) g((x_2y_{(2)}).
\end{align*}
\end{proof}

The following is also stated in \cite[Proposition 3.1.5]{KS} without noetherian hypothesis. Without this hypothesis, it is not true
as showed in Example \ref{counter-ex}.
\begin{prop} \label{dual-of-coPoisson-Hopf}
Let  $(H, \mu , \eta, \Delta ,\varepsilon , S, p)$ be a left or right noetherian Poisson Hopf algebra. Then the Hopf dual $H^\circ$ is a
co-Poisson Hopf algebra.
\end{prop}

\begin{proof} By Proposition \ref{A-o-copoisson}, $H^\circ$ is a
co-Poisson coalgebra. We only need to check the compatible condition \eqref{q(ab)}, that is, for all $f, g \in H^\circ$,
$$ p^*(f * g) = p^*(f)\mu^*(g)+ \mu^*(f)p^*(g).$$
This is true, because for any $x, y \in H$,
\begin{align*}
  & p^*(f * g)(x \otimes y)=(f * g)(\{x, y\})\\
 =&f(\{x_1, y_1\})  \, g(x_2y_2) +f(x_1y_1) \, g(\{x_2, y_2\}),
\end{align*}
and
\begin{align*}
& (p^*(f)\mu^*(g)+ \mu^*(f)p^*(g)) (x \otimes y)\\
=&(f_{(1)}* g_1 \otimes f_{(2)}* g_2) (x \otimes y) + (f_1 * g_{(1)} \otimes f_2 * g_{(2)}) (x \otimes y)\\
=&f_{(1)}(x_1) g_1(x_2) f_{(2)}(y_1) g_2(y_2) + f_1(x_1) g_{(1)}(x_2) f_2(y_1) g_{(2)}(y_2)\\
 =&f(\{x_1, y_1\})  \, g(x_2y_2) +f(x_1y_1) \, g(\{x_2, y_2\}).
\end{align*}
\end{proof}

\begin{ex} \label{counter-ex}
Let $A=k[x_1,\cdots,x_d,\cdots]$ be a polynomial algebra with
infinitely many variables $\{x_i \mid i \geq 1\}$, which is a Hopf algebra viewed as the
enveloping algebra of the abelian Lie algebra $\mathfrak{g}=kx_1
\oplus kx_2 \oplus \cdots \oplus kx_d \oplus \cdots.$ Let
\begin{align*}
\{x_1,x_i\}&=0 \textrm{ for all } i \in \mathbb{N},\\
\{x_i,x_j\}&= \begin{cases}x_1, \quad &j=i+1,\\
0, & \textrm{otherwise}
\end{cases}
\,\, \textrm{for all }  1 <i <j \in \mathbb {N}.
\end{align*}
Then $\{\{x_i, x_j\}, x_k\}=0$ for all $i,j, k$. As in
\cite[Proposition 1.8]{LPV}, $A$ is endowed with a Poisson algebra structure. It is easy to check by induction on the degree of homogeneous elements
that \eqref{poisson-hopf} holds. So,
$A$ is a Poisson Hopf algebra.

 We assert that $A^\circ$ is not a co-Poisson Hopf algebra
by showing that $ \{-,-\}^* ( A^\circ) \nsubseteq A^\circ \otimes
A^\circ$.
Let $f: A \to k$ be the linear map given by
$$
\begin{cases}
f(1_A)=f(x_1)=1_k,\\
f(a)=0\, \textrm{ for all other monic monomials } a \in A.
\end{cases}
$$
Then $I=  \langle \{ x_1^2,x_2,x_3,\cdots\} \rangle \subseteq \ker
f$. Note that $I$ is a cofinite ideal of $A$, and so, $f \in
A^\circ$. Suppose $\{-.-\}^* (f) \in A^\circ \otimes A^\circ$ and
\[\{-.-\}^* (f)= \sum_{i=1}^n g_i \otimes h_i,\]
with $I_i \in \ker g_i$ and $J_i \in \ker h_i$ are all cofinite
ideals of $A$. Then
\[J= \bigcap_{i=1}^n (I_i \cap J_i).\]
is a cofinite ideal of $A$, and $\{-,-\}^* (f) (J \otimes
A)=\sum_{i=1}^n (g_i \otimes h_i) (J \otimes A)=0$. It follows that
$\{J,A\} \subseteq \ker f$.

Since $J$ is cofinite in $A$, $\{x_2+J, x_3+J, \cdots\}$ is linearly
dependent in $A/J$. Then there exists $i \ge 2$  and $\lambda_2,
\cdots, \lambda_{i-1}, \lambda_i \in K$ such that $\lambda_i\neq 0$
and
\[\lambda_2 x_2 + \lambda_3 x_3 + \cdots +\lambda_i x_i \in J.\] Now
\[\{\lambda_2 x_2 + \lambda_3 x_3 + \cdots +\lambda_i x_i ,x_{i+1} \}= \lambda_i x_1 \notin \ker f,\]
which contradicts to $\{J,A\} \subseteq \ker f$.
\end{ex}

\section{Co-Poisson coalgebra structures on $k[x_1, x_2,\cdots, x_d]$}

Suppose $\mathfrak{g}=kx_1 \oplus kx_2 \oplus \cdots \oplus kx_d$ is
a $d$-dimensional Lie algebra. The co-Poisson Hopf structures on $U(\mathfrak{g})$
are in one-to-one correspondence with the Lie bialgebra structures on $\mathfrak{g}$.
If $\mathfrak{g}$ is non-abelian, then there is no nontrivial Poisson Hopf structure on $U(\mathfrak{g})$  (see Proposition \ref{U-of-non-abel}).
So, we turn to consider the case when $\mathfrak{g}$ is abelian from now on. Then
$A=U(\mathfrak{g})=k[x_1,\cdots, x_d]$. We first characterize
co-Poisson coalgebra structures on $k[x_1, x_2,\cdots, x_d]$ in this section.

Suppose $(C, \Delta, \varepsilon)$ is a coalgebra. Let $P(C)$ be the subspace of $C$ consisting of
all primitive elements of $C$. Assume $P(C)= \oplus_{i \in I} ke_i$
as a vector space, with $(I,<)$ being well
ordered by using Well Ordering Principle. Let
$$\mathcal{I} = \bigoplus_{i,j \in I,i <j} k(e_i \otimes e_j -e_j \otimes e_i).$$


\begin{lem}\label{elements-in-i}
Retain the notations above. Suppose $X \in C \otimes C$. Then $X \in
\mathcal{I}$ if and only if
$$(1+t_2)(X)=0\, \textrm{ and } \, (\Delta \otimes 1) (X) =(1-t_3) (1 \otimes X).$$
\end{lem}

\begin{proof}
``$\Rightarrow$'' Trivial.

``$\Leftarrow$'' Assume $0 \neq X=\sum _{i=1}^n a_i \otimes b_i$
with $n$ minimal. Then $\{ a_1, \cdots, a_n\}$ and
$\{b_1,\cdots,b_n\}$ are both linearly independent. Since
$(1+t_2)(X)=0$, i.e., $\sum _{i=1}^n a_i \otimes b_i =-\sum _{i=1}^n
b_i \otimes a_i$, then
$$(1-t_3)(1 \otimes X)=\sum _{i=1}^n (1\otimes a_i \otimes b_i -b_i \otimes 1 \otimes a_i )
= \sum _{i=1}^n(1 \otimes a_i + a_i \otimes 1) \otimes b_i.$$ Since
$(\Delta \otimes 1) (X) =(1-t_3) (1 \otimes X)$, then $\sum _{i=1}^n
\Delta(a_i) \otimes b_i = \sum _{i=1}^n (1 \otimes a_i + a_i \otimes
1) \otimes b_i$. Thus $a_i \in P(C)$ by the independence of the
$b_i$'s.

By using $\sum_{i=1}^na_i \otimes b_i =- \sum_{i=1}^nb_i \otimes a_i$
and a similar discussion, we have $b_i \in P(C)$ as well. The
assertion follows.
\end{proof}

The following facts are obvious.

\begin{lem}\label{fact}
Let $B$ be a bialgebra. Then
\begin{enumerate}
\item If $X \in B \otimes B$ is skew-symmetric, then so is $X \Delta(x)$ for any $x \in P(B)$.
\item For any $a \in B$ and $X \in B \otimes B$,
$$(\Delta \otimes 1) (X \Delta (a)) =(\Delta \otimes 1) (X) \cdot \Delta^{(2)}(a).$$
\end{enumerate}
\end{lem}

In the following, $A=U(\mathfrak{g})=k[x_1, x_2, \cdots,
x_d]$.
Let $\mathcal{H}(A)$
be the set of all monic monomials of $A$. For any $a \in \mathcal{H}(A)$, $\Delta(a)=\sum a_1 \otimes a_2$ is always assumed to be
the expression by the standard $k$-basis of $k[x_1, x_2, \cdots,
x_d]$. For any $a \in \mathcal{H}(A)$, $|a|$ is the degree of $a$.

First, we establish a reciprocity law for two linear maps from $A$ to $A \otimes A$, which is a key step to characterize
the co-Poisson structures.

\begin{prop}\label{main2}
Let $q: A \to A
\otimes A$ and $I: A \to A \otimes A$  be two linear maps. Then $ I(a)
= (-1)^{|a_2|}q(a_1) \Delta (a_2)$ for all $a \in A$ if and only if
$q(a)=I(a_1) \Delta(a_2)$ for all $a \in A$.
\end{prop}

\begin{proof}
 First note that in our case (the algebra is generated by primitive elements), for any $1 \neq a \in \mathcal{H}(A)$,
\begin{equation}\label{delta-a1-a2}
  (-1)^{|a_2|} \Delta(a_1a_2)=0.
\end{equation}
Since $a
\otimes 1$ is one of the terms of $\Delta(a)$ for any $a \in \mathcal{H}(A)$, then
\begin{equation}\label{delta-a}
 \sum_{a_2 \neq 1} (-1)^{|a_2|+1} \Delta(a_1a_2) =\Delta(a).
\end{equation}

``$\Rightarrow$'' Obviously, $I(1)=q(1)$, and so $q(a)=I(a_1)
\Delta(a_2)$ for $a=1$. We prove $q(a)=I(a_1) \Delta(a_2)$ holds for
any $a \in \mathcal{H}(A)$ by induction on the degree of $a$.
Suppose $q(a)=I(a_1) \Delta(a_2)$ holds for all $a \in
\mathcal{H}(A)$ of degree no more than $n$. To finish the proof, it
suffices to show that $q(ax)=I(a_1x)\Delta(a_2) + I(a_1) \Delta(a_2
x)$ for any $x \in P(A)$. Since, by assumption,
\begin{align*}
&I(ax)=(-1)^{|a_2|} q(a_1x) \Delta(a_2)
+ (-1)^{|a_2|+1} q(a_1) \Delta(a_2x)\\
=&q(ax)+ \sum_{a_3 \neq 1} (-1)^{|a_3|} (I(a_1x) \Delta(a_2)+
I(a_1)\Delta(a_2x)) \Delta(a_3) \\
&+ (-1)^{|a_3| +1} I(a_1) \Delta(a_2) \Delta(a_3x),
\end{align*}
then
$$I(ax) + I(a_1) \Delta(a_2x) = q(ax) +\sum_{a_3 \neq 1} (-1)^{|a_3|} I(a_1x) \Delta(a_2a_3).$$
Thus $q(ax)= I(a_1x) \Delta(a_2) + I(a_1) \Delta(a_2x)$ by
\eqref{delta-a}.

``$\Leftarrow$'' Obviously, $I(a) = (-1)^{|a_2|} q(a_1) \Delta (a_2)$
holds for $a =1$ as $q(1) = I(1)$ by $q(a)=I(a_1) \Delta(a_2)$. It
suffices to show that if $I(a) = (-1)^{|a_2|} q(a_1) \Delta (a_2)$
then for any $x \in P(A)$,
$$I(ax)=(-1)^{|a_2|} q(a_1x) \Delta(a_2) +(-1)^{|a_2|+1} q(a_1) \Delta(a_2x).$$
This is equivalent to that
$$I(ax)=(-1)^{|a_3|} (I(a_1x) \Delta(a_2)+I(a_1) \Delta(a_2x)) \Delta(a_3) +
(-1)^{|a_3|+1} I(a_1) \Delta(a_2)\Delta(a_3x),$$ i.e.,
$$I(ax)=(-1)^{|a_3|} I(a_1x) \Delta(a_2a_3),$$
which is always true by \eqref{delta-a1-a2}.
\end{proof}


To prove Proposition \ref{main1}, we need the following lemma.
\begin{lem}\label{turn}
Let $A=U(\mathfrak{g})=k[x_1,\cdots,x_d]$. Then for any linear map $q: A \to A
\otimes A$ and $a \in \mathcal{H}(A)$,
$$  (-1)^{|a_2|} (1 \otimes q) \Delta(a_1) \cdot \Delta^{(2)}(a_2) = (-1)^{|a_2|} 1 \otimes q(a_1) \Delta (a_2).$$
\end{lem}

\begin{proof}
We claim first for any $a \in \mathcal{H}(A)$,
\begin{equation}\label{equ-1}
 \sum (-1)^{|a_1|+|a_2|} a_1 a_3 \otimes a_2 \otimes a_4= \sum (-1)^{|a_1|} 1 \otimes a_1 \otimes a_2.
\end{equation}
It is obviously true for $a=1$. Now assume \eqref{equ-1} holds
for $a$. We show \eqref{equ-1} holds for $ax$ for any $x \in P(A)$.

Since $\Delta^{(3)}(x) =1 \otimes 1 \otimes 1 \otimes x+1 \otimes 1
\otimes x \otimes 1+1 \otimes x \otimes 1 \otimes 1+x \otimes 1
\otimes 1 \otimes 1,$ then
\begin{align*}
&(-1)^{|(ax)_1|+|(ax)_2|} (ax)_1 (ax)_3 \otimes (ax)_2 \otimes (ax)_4\\
= &(-1)^{|a_1|+|a_2|} a_1 a_3 \otimes a_2 \otimes a_4x+(-1)^{|a_1|+|a_2|} a_1 a_3x \otimes a_2 \otimes a_4\\
 &-(-1)^{|a_1|+|a_2|} a_1 a_3 \otimes a_2x \otimes a_4-(-1)^{|a_1|+|a_2|} a_1 x a_3 \otimes a_2 \otimes a_4\\
= &(-1)^{|a_1|+|a_2|} a_1 a_3 \otimes a_2 \otimes a_4x-(-1)^{|a_1|+|a_2|} a_1 a_3 \otimes a_2x \otimes a_4\\
=&(-1)^{|a_1|} 1 \otimes a_1 \otimes a_2x - (-1)^{|a_1|} 1 \otimes a_1x \otimes a_2\\
= &(-1)^{|(ax)_1|} 1 \otimes (ax)_1 \otimes (ax)_2.
\end{align*}
Hence \eqref{equ-1} holds for all $a \in \mathcal{H}(A)$.
By applying $1 \otimes q \otimes \Delta$ to \eqref{equ-1}, then
$$(-1)^{|a_1|+|a_2|} a_1 a_3 \otimes q(a_2) \otimes \Delta (a_4)=  (-1)^{|a_1|} 1 \otimes q(a_1) \otimes  \Delta(a_2).$$
Thus
$$  (-1)^{|a_1|} (1 \otimes q) \Delta(a_1) \cdot \Delta^{(2)}(a_2) = (-1)^{|a_1|} 1 \otimes q(a_1) \Delta (a_2),$$
and the proof is finished.
\end{proof}

\begin{prop}\label{main1}
Let $A=U(\mathfrak{g})=k[x_1,\cdots,x_d]$ and $ \mathcal{I} =
\oplus_{i<j } k(x_i \otimes x_j-x_j \otimes x_i).$ Then, for any linear map
$q: A \to A \otimes A$, the following are equivalent.
\begin{enumerate}
\item $q$ is skew-symmetric and satisfies the co-Leibniz rule.
\item $ I(a)=(-1)^{|a_2|} q(a_1) \Delta(a_2) \in \mathcal{I}$ for all $a \in A$.
\end{enumerate}
\end{prop}

\begin{proof}
``(1) $\Rightarrow$ (2)'' By Lemma \ref{elements-in-i}, it suffice to show
$(1+t_2) \left((-1)^{|a_2|} q(a_1) \Delta(a_2)\right)=0$ and
$(\Delta \otimes 1) ((-1)^{|a_2|} q(a_1) \Delta(a_2)) =(1-t_3)\left(1 \otimes (-1)^{|a_2|} q(a_1) \Delta(a_2)\right).$

Note that $A$ is generated by primitive elements. It follows from
Lemma \ref{fact} that $(1+t_2) ((-1)^{|a_2|} q(a_1) \Delta(a_2))=0$
for all $a \in \mathcal{H}(A)$.

By Lemmas \ref{fact}, \ref{turn} and the co-Leibniz rule \eqref{co-Leibnitz-2},
\begin{align*}
&(\Delta \otimes 1) ((-1)^{|a_2|} q(a_1) \Delta(a_2))\\
= & (-1)^{|a_2|} (\Delta \otimes 1) q(a_1) \cdot \Delta^{(2)}(a_2)\\
=&(1-t_3)(-1)^{|a_2|}(1 \otimes q) \Delta(a_1) \cdot \Delta^{(2)}(a_2)\\
=&(1-t_3)\left((-1)^{|a_2|}(1 \otimes q) \Delta(a_1) \cdot \Delta^{(2)}(a_2)\right)\\
=&(1-t_3)\left(1 \otimes (-1)^{|a_2|} q(a_1) \Delta(a_2)\right).
\end{align*}

``(2) $\Rightarrow$ (1)'' Suppose $ (-1)^{|a_2|} q(a_1) \Delta(a_2) \in
\mathcal{I}$ for all $a \in \mathcal{H}(A)$. Then $q(1) \in \mathcal{I}$.
We check the skew symmetric property of $q$ and the co-Leibniz rule
$$(\Delta \otimes 1) q (a)=(1-t_3) (1 \otimes q) \Delta(a)$$ by induction on the degree of
$a$.
They are true for $a=1$ by Lemma \ref{elements-in-i}. Assume they are true
for $a$ of degree no more than $n$. Now for any $a$ with $\deg
a=n+1$.
Since $I(a)= (-1)^{|a_2|} q(a_1) \Delta(a_2) \in \mathcal{I}$,
$$q(a) = \sum_{a_2 \neq 1} (-1)^{|a_2|+1} q(a_1) \Delta(a_2) + I(a).$$
Since $q(a_1)$ is skew symmetric for $a_2 \neq 1$ by induction
hypothesis and $I(a) \in \mathcal{I}$,  then $q(a)$ is skew
symmetric by Lemma \ref{fact}. By Lemma \ref{turn},
\begin{align*}
&(\Delta \otimes 1)( q(a))\\
=&(\Delta \otimes 1) \left( \sum_{a_2 \neq 1} (-1)^{|a_2|+1} q(a_1) \Delta(a_2) + I(a)\right)\\
=&(-1)^{|a_2|+1} \sum_{a_2 \neq 1} (\Delta \otimes 1) q(a_1) \cdot \Delta^{(2)}(a_2) +   (\Delta \otimes 1)(I(a))\\
=&(1-t_3) \left( \sum_{a_2 \neq 1}(-1)^{|a_2|+1}(1 \otimes q) \Delta(a_1) \right)  \cdot \Delta^{(2)} (a_2) + (1-t_3)(1 \otimes I(a))\\
=&(1-t_3) \left( \sum_{a_2 \neq 1}(-1)^{|a_2|+1}(1 \otimes q) \Delta(a_1) \cdot \Delta^{(2)} (a_2) \right)\\
&+ (1-t_3) \left(\sum (-1)^{|a_2|}(1 \otimes q(a_1)\Delta(a_2))\right)\\
=&(1-t_3) \left( (1 \otimes q) \Delta(a) \right).
\end{align*}
The proof is finished.
\end{proof}

\begin{prop} \label{co-jac-formula} For any $a  \in \mathcal{H}(A)$,
let $I(a) = \sum_{1 \le i,j \le d} \lambda_a^{ij} x_i \otimes x_j
\in \mathcal{I}$ with $(\lambda_a^{ij})_{d \times d} \in M_d(k)$
skew-symmetric. Then the linear map $q: A \to A \otimes A, a \mapsto
I(a_1)\Delta(a_2)$ defines a co-Poisson structure on the coalgebra
$A$ if and only if for all $1 \le i<j<k \le d$ and $a \in A$,
\begin{equation}\label{co-Jaco-id}
\sum_{s=1}^d  \left( \lambda_{a_1}^{sk} \lambda_{x_sa_2}^{ij}
+ \lambda_{a_1}^{si} \lambda_{x_sa_2}^{jk}+\lambda_{a_1}^{sj} \lambda_{x_sa_2}^{ki} \right)=0.
\end{equation}
\end{prop}

\begin{proof} By Proposition \ref{main1}, we need only to care for the co-Jacobi identity.
For any $a \in \mathcal{H}(A)$,
\begin{align*}
&(q \otimes 1) q(a)=(q \otimes 1) (I(a_1) \Delta(a_2))\\
=&(q \otimes 1) (\sum_{s,t}\lambda_{a_1}^{s t} (x_s \otimes x_t)
\Delta(a_2))
=\sum_{s,t} \lambda_{a_1}^{s t} q(x_s a_2) \otimes x_t a_3\\
=&\sum_{s,t} \lambda_{a_1}^{s t} ( I(x_s a_2) \Delta (a_3)+ I(a_2) \Delta (x_s a_3)) \otimes x_t a_4\\
=&\sum_{s,t,i,j} \lambda_{a_1}^{s t} \lambda_{x_sa_2}^{i j} (x_i \otimes x_j \otimes x_t) (a_3 \otimes a_4 \otimes a_5)\\
&+\sum_{s,t,i,j} \lambda_{a_1}^{s t} \lambda_{a_2}^{i j}(x_i x_s \otimes x_j \otimes x_t
+x_i \otimes x_j x_s \otimes x_t) (a_3 \otimes a_4 \otimes a_5)\\
=&U+V,
\end{align*}
where $U$ is the first term, i.e.,
$$U=\sum_{s,t,i,j} \lambda_{a_1}^{s t} \lambda_{x_sa_2}^{i j}
(x_i \otimes x_j \otimes x_t) (a_3 \otimes a_4 \otimes a_5),$$ and $V$ is the second term,
i.e.,
$$V=\sum_{s,t,i,j} \lambda_{a_1}^{s t} \lambda_{a_2}^{i j}(x_i x_s \otimes x_j \otimes x_t
+x_i \otimes x_j x_s \otimes x_t) (a_3 \otimes a_4 \otimes a_5).$$
Then, by the skew-symmetric property of $\{ \lambda_a^{ij} \}_{d \times d}$
and the cocommutativity of $A$,
\begin{align*}
\circlearrowleft V
=&\sum_{s,t,i,j} (\lambda_{a_1}^{s t} \lambda_{a_2}^{i j}+ \lambda_{a_1}^{i j} \lambda_{a_2}^{t s})
(x_i x_s \otimes x_j \otimes x_t) (a_3 \otimes a_4 \otimes a_5)\\
& + \sum_{s,t,i,j} (\lambda_{a_1}^{s t} \lambda_{a_2}^{i j}+ \lambda_{a_1}^{j i} \lambda_{a_2}^{s t})
(x_i \otimes x_j x_s \otimes x_t) (a_3 \otimes a_4 \otimes a_5)\\
& +  \sum_{s,t,i,j} (\lambda_{a_1}^{s t} \lambda_{a_2}^{i j}+ \lambda_{a_1}^{j i} \lambda_{a_2}^{s t})
(x_t \otimes x_i \otimes x_j x_s ) (a_3 \otimes a_4 \otimes a_5)\\
=& 0.
\end{align*}
So, the co-Jacobi identity holds if and only if $\circlearrowleft U = 0$.
Note that
\begin{align*}
\circlearrowleft U
=& \sum_{s,t,i,j} \lambda_{a_1}^{s t} \lambda_{x_sa_2}^{i j}
(x_i \otimes x_j \otimes x_t  + x_t \otimes x_i \otimes x_j + x_j \otimes x_t \otimes x_i) (a_3 \otimes a_4 \otimes a_5)\\
=&\sum_{i,j,k}\sum_s  \left( \lambda_{a_1}^{sk} \lambda_{x_sa_2}^{ij}
+ \lambda_{a_1}^{si} \lambda_{x_sa_2}^{jk}+\lambda_{a_1}^{sj} \lambda_{x_sa_2}^{ki} \right)
(x_i \otimes x_j \otimes x_k)(a_3 \otimes a_4 \otimes a_5).
\end{align*}
If $\circlearrowleft U = 0$,  then, by considering the coefficients of elements in degree 3
in $\circlearrowleft U$ (i.e., when $a_3 \otimes a_4 \otimes a_5 = 1 \otimes 1 \otimes 1$),
$$c_{ijk}=\sum_s \left( \lambda_{a_1}^{sk} \lambda_{x_sa_2}^{ij}
+ \lambda_{a_1}^{si} \lambda_{x_sa_2}^{jk}+\lambda_{a_1}^{sj} \lambda_{x_sa_2}^{ki} \right)=0$$
for all $1 \le i, j, k \le d$ and $a \in \mathcal{H}(A)$.

Conversely, if $c_{ijk}=0$ for all $1 \le i, j, k \le d$ and $a \in \mathcal{H}(A)$, then $\circlearrowleft U = 0$.
\end{proof}

In summary, we have the following result.

\begin{thm} \label{main-result}
Let $A=U(\mathfrak{g})=k[x_1,\cdots,x_d]$. Then a linear map
$q: A \to A \otimes A$ gives a co-Poisson coalgebra structure on $A$ if and only if there is a linear map $I : A \to  A \otimes A$ such that for all  $a \in \mathcal{H}(A)$,
\begin{enumerate}
\item $q(a)=I(a_1) \Delta(a_2)$.
\item $I(a)= \sum_{1 \le i,j \le d} \lambda_a^{ij} x_i \otimes x_j \in \mathcal{I}$
with $(\lambda_a^{ij})_{d \times d} \in M_d(k)$
skew-symmetric.
\item For all $1 \le i<j<k \le d$,
$
\sum_{s=1}^d  \left( \lambda_{a_1}^{sk} \lambda_{x_sa_2}^{ij}
+ \lambda_{a_1}^{si} \lambda_{x_sa_2}^{jk}+\lambda_{a_1}^{sj} \lambda_{x_sa_2}^{ki} \right)=0.$
\end{enumerate}
\end{thm}

The following proposition shows that the co-Jacobi identity holds trivially in two variables case.
So, the co-Poisson coalgebra structures on $k[x, y]$ are given by the linear maps $I: k[x, y] \to k(x \otimes y - y \otimes x)$.

\begin{prop} \label{2-variable-case}
Let $A=k[x,y]$. Then there is an one-to-one correspondence between
the co-Poisson coalgebra structures on $A$ and the linear maps $A \to
\mathcal{I}= k(x \otimes y - y \otimes x)$, given by
$$
(q: A \to A \otimes A)  \mapsto (I: A \to \mathcal{I}, a \mapsto
(-1)^{|a_2|} q(a_1) \Delta(a_2),$$
with the inverse map
$(I: A \to \mathcal{I})  \mapsto (q: A \to A \otimes A, a \mapsto
I(a_1) \Delta(a_2)).$
\end{prop}

\begin{proof}
We only need to show the co-Jacobi identity always holds in this
case. Assume $I(a)=l_a(x \otimes y -y \otimes x)$ with $l_a \in k$.
For any $a \in A$,
$$q(a)=I(a_1) \Delta(a_2)=l_{a_1} (x \otimes y -y \otimes x) (a_2 \otimes a_3)=l_{a_1}
(xa_2 \otimes y a_3-ya_2 \otimes xa_3).$$
Then
$\circlearrowleft (q \otimes 1) q(a)= \circlearrowleft l_{a_1}
\left(q(xa_2) \otimes ya_3 -q(ya_2) \otimes xa_3\right).$ Now
\begin{align*}
&\circlearrowleft l_{a_1} q(xa_2) \otimes ya_3 \\
=& \circlearrowleft l_{a_1} \left( I(xa_2) \Delta (a_3) + I(a_2) \Delta (xa_3)\right) \otimes ya_4\\
=&\circlearrowleft l_{a_1} l_{xa_2} (x \otimes y -y \otimes x) (a_3 \otimes a_4)\otimes ya_5\\
& + \circlearrowleft l_{a_1} l_{a_2} (x \otimes y -y \otimes x) (xa_3 \otimes a_4 + a_3 \otimes xa_4)\otimes ya_5\\
=&\circlearrowleft l_{a_1} l_{xa_2} (xa_3 \otimes ya_4 \otimes ya_5 -ya_3 \otimes xa_4 \otimes ya_5)\\
&+\circlearrowleft l_{a_1} l_{a_2} ( x^2a_3 \otimes ya_4 \otimes ya_5 + xa_3 \otimes xya_4 \otimes ya_5\\
&-xya_3 \otimes xa_4 \otimes ya_5-ya_3 \otimes x^2a_4 \otimes ya_5)\\
=&\circlearrowleft l_{a_1} l_{a_2} (xa_3 \otimes xya_4 \otimes
ya_5-xya_3 \otimes xa_4 \otimes ya_5).
\end{align*}
Similarly,
$$\circlearrowleft l_{a_1} q(ya_2) \otimes xa_3=\circlearrowleft l_{a_1} l_{a_2}
(ya_3 \otimes xya_4 \otimes xa_5-xya_3 \otimes ya_4 \otimes xa_5).
$$
Thus $\circlearrowleft (q \otimes 1) q(a)=0$ for all $a \in
A=k[x,y]$.
\end{proof}

\section{(Co-)Poisson Hopf structures on $k[x_1, x_2,\cdots, x_d]$}

\subsection{Poisson Hopf structures}

As proved in \cite[Proposition 1.8]{LPV}, any Poisson structure on the polynomial algebra $A=k[x_1, x_2,\cdots, x_d]$ is
given by $\{x_i,x_j\}=f_{ij}$ where $\{f_{ij}\}_{d \times d}$ is a skew-symmetric matrix over $A$ such that
for all $1 \le i <j< k \le d$,
\begin{equation}\label{poi-alg-str}
\sum_{l=1}^d \left( f_{lk} \frac{\partial f_{ij}}{ \partial x_l} + f_{li}
\frac{\partial f_{jk}}{ \partial x_l}+ f_{lj} \frac{\partial f_{ki}}{ \partial x_l} \right)=0.
\end{equation}

If $A=U(\mathfrak{g})=k[x_1, x_2, \cdots, x_d]$ is viewed as a Hopf algebra, then the Poisson algebra structures on $A$ can be described in a form dual to
Theorem \ref{main-result}. The following is a reciprocity law for linear maps $ A \otimes A \to A$.

\begin{lem}\label{main2'}
Let $p: A \otimes A \to A$ and $J: A \otimes A \to A$  be two linear maps. Then $p(a \otimes b)=J(a_1 \otimes b_1) a_2b_2$ for all $a,b \in A$ if and only if
$J(a \otimes b)=(-1)^{|a_2|+|b_2|} p(a_1 \otimes b_1)a_2b_2$ for all $a,b \in A$.
\end{lem}

\begin{proof}
Similar to that of Proposition \ref{main2}.
\end{proof}

\begin{prop}\label{poisson-structure}
Let $A=U(\mathfrak{g})=k[x_1,\cdots,x_d]$. Then a linear map
$p: A \otimes A \to A$ gives a Poisson structure on $A$ if and only if there is a linear map $J : A \otimes A \to A$ such that for all $a,b \in A$,
\begin{enumerate}
\item $p(a \otimes b)=J(a_1 \otimes b_1) a_2b_2.$
\item $J$ is skew-symmetric, and $J(a \otimes b)=0$ except both $a$ and $b$ are of degree $1$. 
\item The Jacobi identity holds for $J$.
\end{enumerate}
\end{prop}

Actually, Proposition \ref{poisson-structure} is exactly \cite[Proposition 1.8]{LPV}. If $p: A \otimes A \to A$ is a Poisson algebra structure on $A$, and $J(a \otimes b)= (-1)^{|a_2|+|b_2|} p(a_1 \otimes b_1) a_2b_2$,
then  $J$ satisfies conditions (2) and (3) in Proposition \ref{poisson-structure}. In this case, condition (1) is the same as
$$\{f,g\}= \sum _{1 \le i,j \le d}\frac{\partial f}{\partial x_i}  \frac{\partial g}{\partial x_j} \{x_i,x_j\}.$$

The Poisson Hopf structures on $A$ are classified in the following
proposition. They are exactly linear Poisson structures on $A$.

\begin{prop}\label{poisson-hopf-str}
Any Poisson Hopf structure on $A= k[x_1, x_2,\cdots, x_d]$ is given by
$$\{ x_i ,x_j\}= \sum_{l=1}^d \lambda^{ij} _l x_l \qquad (1 \leq i, j \leq d),$$
where $\lambda^{ij}_l = -\lambda^{ji}_l$,
subject to the relations, for all $1 \le i <j<k \le d$ and all $1
\le s \le d$,
$$ \sum_{l=1}^d \left(  \lambda^{ij}_l \lambda^{lk}_s +     \lambda^{jk}_l \lambda^{li}_s
+  \lambda^{ki}_l \lambda^{lj}_s   \right)=0.$$
\end{prop}

\begin{proof}
Suppose $\{-,-\}$ is a Poisson Hopf structure on $A$. Then, by \eqref{poisson-hopf},
$$\Delta \left( \{x_i ,x_j\} \right) = 1 \otimes \{x_i,x_j\} +\{x_i ,x_j\} \otimes 1,$$
that is, $\{x_i ,x_j\}$ is a primitive element of $A$. Hence $\{ x_i
,x_j\}= \sum_{l=1}^d \lambda^{ij} _l x_l$ for some $\lambda^{ij} _l
\in k$ such that  $\lambda^{ij} _l= -\lambda^{ji} _l$ for all $l$.
Then \eqref{poi-alg-str} is equivalent to, for all $1 \le i < j < k \le d$
and all $1 \le s \le d$,
$$ \sum_{l=1}^d \left(  \lambda^{ij}_l \lambda^{lk}_s + \lambda^{jk}_l \lambda^{li}_s
+ \lambda^{ki}_l \lambda^{lj}_s \right)=0.$$

Conversely, any such a Poisson algebra structure is in fact a Poisson Hopf structure on $A$.
We need to check that $\Delta \left( \{a,b\} \right)=a_1 b_1 \otimes \{a_2,b_2\} +
\{a_1,b_1\} \otimes a_2 b_2$ for all $a,b \in \mathcal{H}(A)$, which can be done by induction.
\end{proof}

\subsection{Co-Poisson Hopf structures}

Next we discuss co-Poisson Hopf structures on $A$. We fix some
notations here. For any $a=x_1^{n_1} x_2 ^{n_2} \cdots x_d ^{n_d}
\in \mathcal{H}(A)$, we denote by $a!=n_1 ! \cdots n_d !$ and
$a^{(i)}=n_i$ for $1 \le i \le d$. For any $b= x_1^{m_1} x_2 ^{m_2}
\cdots x_d ^{m_d} \in \mathcal{H}(A)$, if $b \mid a$, we denote by
\[ {a \choose b}= {n_1 \choose m_1} {n_2 \choose m_2} \cdots {n_d \choose m_d};\]
if $b \nmid a$, we set ${a \choose b}=0$.

\begin{prop}\label{main3}
Let $A=U(\mathfrak{g})=k[x_1, x_2,\cdots,x_d]$. Then a linear map
$q: A \to A \otimes A$ gives a co-Poisson Hopf algebra structure on $A$ if and only if there is a linear map $I : A \to  A \otimes A$ such that for all  $a \in \mathcal{H}(A)$,
\begin{enumerate}
\item $q(a)=I(a_1) \Delta(a_2)$.
\item For all $1 \le s \le d$, $I(a)=0$ if $a \neq x_s$ and $I(x_s)= \sum _{1 \le i,j \le d} \lambda_s^{ij} x_i \otimes x_j \in \mathcal{I}$ with $(\lambda_s^{ij})_{d \times d} \in M_d(k)$
skew-symmetric.
\item For all $0 \le i < j< k \le d$, $1 \le s \le d$, $ \sum_{l=1}^d \left( \lambda_s^{lk} \lambda_l^{ij}+\lambda_s^{li} \lambda_l^{jk}
+\lambda_s^{lj} \lambda_l^{ki} \right)=0$.
\end{enumerate}
\end{prop}

\begin{proof}
``$\Rightarrow$'' Let $I: A \to A \otimes A$ be the map $I(a)
= (-1)^{|a_2|}q(a_1) \Delta (a_2)$ for all $a \in A$. Then, by Proposition \ref{main2}, $q(a)=I(a_1) \Delta(a_2)$, i.e., (1) holds.

Since $q(ab)=q(a) \Delta (b)+ \Delta(a)q(b)$ for
all $a ,b \in A$, $I(1)=q(1)=0$. Thus $q(x_i)=I(x_i)$ for all $1 \le
i \le d$. Then, by induction and \eqref{q(ab)}, for any $a \in
\mathcal{H}(A)$,
\[ q(a)= \sum_{l=1}^d {a \choose x_l} I(x_l) \Delta \left(  \frac{a}{x_l}\right).\]
If $a \in \mathcal{H}(A)$ with $\deg a =2$, then $I(a)=0$ by
$q(a)=I(a_1) \Delta(a_2)$. By induction, we see $I(a)=0$ for any $a
\in \mathcal{H}(A)$ of degree $\geq 3$. By Proposition \ref{main1}, (2) holds.

In this case, the equation \eqref{co-Jaco-id} is equivalent to
$ \sum_{l=1}^d \left( \lambda_s^{lk} \lambda_l^{ij}+\lambda_s^{li} \lambda_l^{jk}
+\lambda_s^{lj} \lambda_l^{ki} \right)=0,$ i.e., (3) holds.

``$\Leftarrow$'' If $I(a)=0$ for all $x_i \neq a \in
\mathcal{H}(A)$, then $q(a)=I(a_1) \Delta(a_2)$ becomes
\[ q(a)= \sum_{l=1}^d {a \choose x_l} I(x_l) \Delta \left(  \frac{a}{x_l}\right).\]
Now it is easy to check  $q(ab)=q(a) \Delta (b)+ \Delta(a)q(b)$ for all $a,b \in \mathcal{H}(A)$.
\end{proof}

%

%

\begin{ex} Let $A=k[x,y]$. Then there is an one-to-one correspondence between
the co-Poisson Hopf structures on $A$ and the set $\{(I(x), I(y)) \mid I(x), I(y) \in k(x \otimes y -y \otimes x)\}$, given by
\begin{align*}
 q & \mapsto (q(x),q(y)), \, \textrm{and}\\
(I(x),I(y)) & \mapsto q: x^ny^m \mapsto nI(x) \Delta(x^{n-1}y^m) + mI(y)
\Delta(x^ny^{m-1}).
\end{align*}
\end{ex}

\begin{ex} Let $A=k[x_1,x_2,x_3]$. Then any Poisson Hopf structure on $A$ is given by
$$\{x_1,x_2\}= \lambda_{12}^1 x_1 + \lambda_{12}^2 x_2 + \lambda_{12}^3 x_3,$$
$$\{x_2,x_3\}= \lambda_{23}^1 x_1 + \lambda_{23}^2 x_2 + \lambda_{23}^3 x_3,$$
$$\{x_3,x_1\}= \lambda_{31}^1 x_1 + \lambda_{31}^2 x_2 + \lambda_{31}^3 x_3,$$
subject to the relations,
$$\lambda_{12}^1 \lambda_{31}^l +
\lambda_{23}^2 \lambda_{12}^l+\lambda_{31}^3 \lambda_{23}^l = \lambda_{12}^2 \lambda_{23}^l
+\lambda_{23}^3 \lambda_{31}^l + \lambda_{31}^1 \lambda_{12}^l, \,\,
\textrm{for all} \,\, 1 \le l \le 3.
$$

Any co-Poisson Hopf structure  on $A$ is given by
$$q(x_1) =\lambda_1^{12} (x_1 \otimes x_2 -x_2 \otimes x_1)+ \lambda_1^{23}
(x_2 \otimes x_3 -x_3 \otimes x_2) + \lambda_1^{31}
(x_3 \otimes x_1 -x_1 \otimes x_3),$$
$$q(x_2) =\lambda_2^{12} (x_1 \otimes x_2 -x_2 \otimes x_1)+ \lambda_2^{23}
(x_2 \otimes x_3 -x_3 \otimes x_2) + \lambda_2^{31} (x_3 \otimes x_1 -x_1 \otimes x_3),$$
$$q(x_3) =\lambda_3^{12} (x_1 \otimes x_2 -x_2 \otimes x_1)+ \lambda_3^{23}
(x_2 \otimes x_3 -x_3 \otimes x_2) + \lambda_3^{31} (x_3 \otimes x_1
-x_1 \otimes x_3),$$ subject to the relations,
$$ \lambda_k^{12} \lambda_1^{31}+\lambda_k^{23} \lambda_2^{12}+\lambda_k^{31} \lambda_3^{23}
=\lambda_k^{31} \lambda_1^{12}+\lambda_k^{12}
\lambda_2^{23}+\lambda_k^{23} \lambda_3^{31}, \,\, \textrm{for any}
\,\, 1 \le k \le 3.$$
\end{ex}

\subsection{Dual (co-)Poisson structures}
It is well-known that the dual algebra $A^*$ of the coalgebra $A=k[x_1, x_2,\cdots,x_d]$ is the algebra of formal divided power series (it is called Hurwitz series in \cite[Proposition 2.4]{Ke} in one variable), which is isomorphic to the
algebra of formal power series $k[[x_1, x_2,\cdots,x_d]]$. By Cartier-Gabriel-Kostant-Milnor-Moore Theorem (\cite[Theorem 8.1.5]{Sw} and \cite[$\S$ 6]{MM}),
the finite dual $A^\circ$ of the polynomial Hopf algebra $A=k[x_1, x_2,\cdots,x_d]$ is isomorphic to
$A \ltimes k^n$ as Hopf algebras,
where $k^n$ carries the additive
group structure (the group-like elements in the Hopf dual). Similar to \cite[Proposition 1.8]{LPV}, the following lemma holds.

\begin{lem}
Let $\tilde{A}=k[[x_1,\cdots,x_d]]$ be the algebra of formal power series. Any Poisson structure on  $\tilde{A}$ is given by
$\{x_i, x_j\} =f_{ij}$, where $(f_{ij})_{d \times d}$ is a skew-symmetric matrix over $\tilde{A}$ such that for all $1 \le i <j<k \le d$,
\[ \{ \{x_i,x_j\},x_k\}+ \{ \{x_j,x_k\},x_i\}+ \{ \{x_k,x_i\},x_j\}=0.\]
In this case, for all $f, g \in \tilde{A}$,
$ \{f,g\}= \sum_{i,j=1}^d \frac{ \partial f}{ \partial x_i} \frac{ \partial g}{ \partial x_j} f_{ij}.$
\end{lem}

\begin{thm} \label{main5} Suppose $\charac k =0$. Let $\tilde{A}=k[[x_1, x_2,\cdots,x_d]]$  be the algebra of formal power series
and $A=k[x_1, x_2,\cdots,x_d]$. Then there is an one-to-one correspondence between the Poisson algebra structures on $\tilde{A}$
and the co-Poisson coalgebra structures on $A$.
\end{thm}

\begin{proof} Suppose
$\{x_i,x_j\}= f_{ij}=\sum_{a \in \mathcal{H}(A)} \lambda^{ij}_a a \, (1 \le i,j \le d)$
gives a Poisson algebra structure on $\tilde{A}$. Let $\alpha^{ij}_a= a! \lambda^{ij}_a$ and
\[I(a)= \sum_{1 \le i,j \le d} \alpha^{ij}_a x_i \otimes x_j, \quad q(a)=I(a_1) \Delta (a_2) \, \textrm{for all}\,  a \in \mathcal{H}(A).\]
Then $q: A \to A \otimes A$ is a co-Poisson coalgebra structure on $A$.
In fact, for all $1 \le i<j<k \le d$,
\[\{\{x_i,x_j\},x_k\}= \sum_{a,b \in \mathcal{H}(A)} \sum_{l=1}^d\lambda^{ij}_a \lambda^{lk}_b \frac{\partial a}{ \partial x_l} b= \sum_{c \in \mathcal{H}(A)} \sum^{ab=c}_{ a ,b \in \mathcal{H}(A)} \sum_{l=1}^d  (a^{(l)}+1) \lambda^{ij}_{ax_l} \lambda^{lk}_b c. \]
Then $\circlearrowleft \{\{x_i,x_j\},x_k\}= 0$
if and only if for all $c \in \mathcal{H}(A)$ and all $1 \le i <j<k \le d$,
\[\sum^{ab=c}_{ a ,b \in \mathcal{H}(A)} \sum_{l=1}^d  (a^{(l)}+1)  \left( \lambda^{ij}_{ax_l} \lambda^{lk}_b + \lambda^{jk}_{ax_l} \lambda^{li}_b+ \lambda^{ki}_{ax_l} \lambda^{lj}_b \right)=0.\]
Note that $\alpha^{ij}_a= a! \lambda^{ij}_a$. It is equivalent to
\[\sum^{ab=c}_{ a ,b \in \mathcal{H}(A)} \sum_{l=1}^d  \frac{1}{a! b!}  \left( \alpha^{ij}_{ax_l}\alpha^{lk}_b + \alpha^{jk}_{ax_l}\alpha^{li}_b+ \alpha^{ki}_{ax_l} \alpha^{lj}_b \right)=0.\]
By multiplying it with $c!$, it is easy to see that it is equivalent to
\[ \sum_{l=1}^d \left(  \alpha^{ij}_{c_1 x_l} \alpha^{lk}_{c_2}+\alpha^{jk}_{c_1 x_l} \alpha^{li}_{c_2}+\alpha^{ki}_{c_1 x_l} \alpha^{lj}_{c_2}  \right)=0\]
for all $c \in \mathcal{H}(A)$ and $1 \le i<j<k \le d$.

On the other hand, suppose $q: A \to A \otimes A$ is a co-Poisson coalgebra structure on $A$.
Let
$I(a)=(-1)^{|a_2|} q(a_1) \Delta(a_2):= \alpha^{ij}_a x_i \otimes x_j$,
and $\lambda^{ij}_a= \frac{1}{a!} \alpha^{ij}_a$. Then
\[\{x_i,x_j\}=f_{ij}:= \sum_{a \in \mathcal{H}(A)} \lambda^{ij}_a a.\]
gives a Poisson algebra structure on $\tilde{A}$.
\end{proof}

\begin{rk} A co-Poisson coalgebra structure $q$ on $A$ given by $I: A \to A \otimes A$ is called {\it rational} if there is an integer $n$
such that $I(a)=0$ for all $a \in \mathcal{H}(A)$ with $\deg a \ge n$. Then there is an one-to-one correspondence between Poisson algebra structures on $A$ and rational co-Poisson coalgebra structures on $A$.
\end{rk}

Combining Propositions \ref{poisson-hopf-str} and \ref{main3}, we have the following.

\begin{thm}
There is an one-to-one correspondence between Poisson Hopf structures on $A$ and co-Poisson Hopf structures on $A$. More precisely, assume
\[\{x_i,x_j\}= \lambda^{ij}_1 x_1 + \cdots +\lambda^{ij}_d x_d\]
is a Poisson Hopf structure on $A$. Let
\[I(x_s)= \sum_{1 \le i,j \le d}\lambda^{ij}_s x_i \otimes x_j\]
for all $1 \le s \le d$ and $I(a)=0$ for all other $a \in \mathcal{H}(A)$. Then $q(a)=I(a_1) \Delta(a_2)$ defines a co-Poisson Hopf structure on $A$.
\end{thm}

\section*{Acknowledgments}
The authors thank Ruipeng Zhu for useful discussions. This research is supported by NSFC key project 11331006 and NSFC project
11171067.

\end{document}